\documentclass[12pt]{article}
\usepackage{amsmath,amsthm,amssymb,amsfonts}

\title{Recurrent Rotor-Router Configurations}
\author{Omer Angel \and Alexander E.\ Holroyd}
\date{12 January 2011}

\newtheorem{thm}{Theorem}
\newtheorem{lemma}[thm]{Lemma}
\newtheorem{cor}[thm]{Corollary}
\newtheorem{prop}[thm]{Proposition}

\newcommand{\Z}{\mathbb{Z}}
\newcommand{\df}{\textbf}

\newcounter{mycount}
\newenvironment{ilist}{\begin{list}{\rm(\roman{mycount})}%
   {\usecounter{mycount}\labelwidth=1cm\itemsep -2pt}}{\end{list}}

\begin{document}
\maketitle

\begin{abstract}
We prove the existence of recurrent initial configurations for the rotor walk
on many graphs, including $\Z^d$, and planar graphs with locally finite
embeddings.  We also prove that recurrence and transience of rotor walks are
invariant under changes in the starting vertex and finite changes in the
initial configuration.
\renewcommand{\thefootnote}{}
\footnotetext{Key words and phrases: rotor walk, rotor-router, quasi-random,
recurrence.} \footnotetext{2010 Mathematics Subject Classification: 05C25;
82C20.}
\end{abstract}

\section{Introduction}

The \df{rotor walk} is a derandomized variant of random walk on
a graph $G$, defined as follows.  To each vertex of $G$ we
assign a fixed cyclic order of its neighbours.  This collection
of orders is called the \df{rotor mechanism}.  At each vertex
there is a \df{rotor}: an arrow which can point to any
neighbour.  An assignment of directions to all the rotors is
called a \df{rotor configuration}.  Starting from some rotor
configuration, a \df{particle} is located at some vertex, and
the particle location and rotor configuration evolve together
in discrete time as follows.  At each time step, the rotor at
the particle's current location is incremented to point to the
next neighbour in the cyclic order, and then the particle moves
to this new neighbour.  The rotor walk is obtained by
repeatedly applying this rule.

We will assume throughout that $G$ is an infinite, connected, simple,
undirected graph with all degrees finite.  Given a rotor mechanism, an
initial rotor configuration, and an initial particle location, it is then
easy to see that the rotor walk either visits each vertex infinitely many
times, or visits each vertex finitely many times (see e.g.
\cite[Lemma~6]{h-propp}). We call these two cases \df{recurrent} and
\df{transient} respectively.

In many settings there are striking similarities between the
behaviour of the rotor walk and the expected behaviour of
\emph{random} walk on the same graph; see e.g.\
\cite{cooper-doerr-spencer-tardos-0,cooper-spencer,
doerr-friedrich,h-propp,levine-peres-2}. However, with regard
to recurrence and transience there can also be differences. For
instance, on any $G$ (even a recurrent graph) it is easy to
find an initial rotor configuration that is transient (indeed,
we can arrange for the particle to trace any self-avoiding
path). In the other direction, it was shown in
\cite{landau-levine} that recurrent initial rotor
configurations exist on the infinite binary tree. These matters
have been investigated further on trees in
\cite{a-h,landau-levine}, but much less is known for more
general graphs.  Our goal is this article is to prove the
existence of recurrent rotor configurations in a broad range of
settings, including many where the random walk is transient.

We will also show that recurrence and transience of rotor walk
do not depend on the starting vertex.  Therefore we may refer
to an initial rotor configuration as recurrent or transient
without specifying the starting vertex.

\begin{thm}\label{T:dichotomy}
Fix a graph, a rotor mechanism and an initial rotor
configuration.  The rotor walk is either recurrent for every
starting vertex, or transient for every starting vertex.
\end{thm}

Let $\Z^d$ denote the graph with vertex set $\Z^d$ and an edge
between each pair of vertices at Euclidean distance $1$.

\begin{thm}\label{T:zd}
For any $d\geq 1$ and any rotor mechanism on $\Z^d$, there
exists a recurrent rotor configuration.
\end{thm}

An example was given in \cite{h-propp} of a recurrent rotor
configuration on $\Z^2$. For all $d\geq 3$, Theorem \ref{T:zd}
provides the first proof of the existence of a recurrent
configuration for a translation-invariant rotor mechanism. This
answers a question posed by Jim Propp in 2003 (personal
communication), and stated as an open problem in \cite{h-propp}
and \cite{landau-levine}. Our proof is strongly motivated by a
recent work of Reddy \cite{reddy}, in which it is argued on the
basis of simulation evidence that a certain explicit rotor
configuration on $\Z^3$ is recurrent. Our proof also includes a
simple explicit rotor configuration.

Our construction can be generalized to many other settings, and
in particular we have the following.  A planar embedding of a
planar graph is called \df{locally finite} if every bounded
region of the plane contains only finitely many vertices.
(There exist planar graphs with no locally finite embedding,
such as $\Z^2$ with a singly-infinite path attached to every
vertex.)

\begin{thm}\label{T:planar}
For a planar graph with a locally finite planar embedding, and
a rotor mechanism in which each rotor points to the neighbours
in clockwise or anticlockwise order (possibly with different
directions at different vertices), there exists a recurrent
rotor configuration.
\end{thm}

\begin{thm}\label{T:any}
For any graph, there exists a rotor mechanism that admits a
recurrent rotor configuration.
\end{thm}

For some graphs including $\Z^d$, we can also give recurrent
configurations for which the behaviour of the rotor walk is
extremely regular, as follows. For $x\in\Z^d$, define
\[
\alpha(x):=\begin{cases}
  1& \text{ if $|x_1|,\ldots,|x_d|$ have a unique maximum;}\\
  0& \text{ otherwise.}
\end{cases}
\]

\begin{thm}\label{T:zd-exact}
  For any rotor mechanism on $\Z^d$, there exists a recurrent rotor
  configuration such that for the rotor walk started at $0$, just before the
  $(k+1)$st traversal from $0$ to $(1,0,\dots,0)$, vertex $x$
  has been entered
  exactly $$\Big[2d(k-\|x\|_\infty)+\alpha(x)\Big]^+$$ times.
\end{thm}

For the reader's convenience we present an explicit family of
configurations on $\Z^d$ satisfying the conclusion of Theorem
\ref{T:zd-exact} (see Section~\ref{S:exact} for a more general
result). Let $x\in\Z^d$. If $\alpha(x)=1$, set the rotor at $x$
so that it will {\em next} point to the unique neighbour $y$
such that $\|y\|_\infty<\|x\|_\infty$.  If $\alpha(x)=0$ and
$x\neq 0$, let the rotor at $x$ point towards any neighbour $y$
with $\|y\|_\infty<\|x\|_\infty$.  Set the rotor at $0$
arbitrarily.

\section{Invariance}

In this section we prove Theorem~\ref{T:dichotomy}, and a
corollary thereof.  These results will not be needed for the
proofs of Theorems \ref{T:zd}--\ref{T:zd-exact}.  The proof of
Theorem~\ref{T:dichotomy} uses the Abelian property of the
rotor walk.  We will use a truncation argument in order to
apply a version of the Abelian property for finite graphs
proved in \cite{hlmppw}. An alternative approach would be to
derive a version of the Abelian property that holds for
transfinite-time walks on infinite graphs.  Write $d_x$ for the
degree of vertex $x$.

\begin{proof}[Proof of Theorem~\ref{T:dichotomy}]
It is clearly sufficient to prove that if a rotor configuration
is recurrent for the rotor walk started at some vertex $x$ then
it is also recurrent for the walk started at any neighbour $y$
of $x$. For any $m\geq0$, we show that the rotor walk started at $y$ visits
$x$ at least $m$ times.

Let $S$ be the finite set of vertices visited by the rotor walk
started at $x$ until it has made $d_x+m$ returns to $x$.  Let
$F$ be the subgraph of $G$ induced by $S$, modified as follows:
add a sink vertex $z$, and replace each edge of $G$ leaving $S$
by a directed edge to $z$; also split $x$ into two vertices
$x_+$ and $x_-$, and split each edge incident to $x$ into a
directed edge from $x_+$ and a directed edge to the sink $x_-$.
By the Abelian property, \cite[Lemma 3.9]{hlmppw}, if we start
$d_x+m$ rotor particles at $x_+$, and let them perform rotor
walks until they reach the set of sinks $\{z,x_-\}$, then they
all in fact reach $x_-$, regardless of the order in which they
move.

Start $d_x+m$ particles at $x_+$ and move them in $F$ in the following order.
First, let $d_x$ particles each take one step.  This leaves the rotor
configuration unchanged, and one particle at each neighbour of $x$ (including
$y$).  Now let the particle at $y$ perform rotor walk.  It will follow
exactly the trajectory of the rotor walk in $G$ started at $y$.  Each time it
is absorbed at $x_-$, continue with one of the $m$ particles remaining at
$x_+$. It follows that the rotor walk from $y$ in $G$ visits $x$ at least $m$
times before leaving $S$.
\end{proof}

One consequence of Theorem~\ref{T:dichotomy} is that recurrence
and transience are also insensitive to local changes in the
configuration.

\begin{cor}
Fix a graph and a rotor mechanism.  If two rotor configurations
differ only at finitely many vertices, then they are either
both recurrent or both transient.
\end{cor}

\begin{proof}
It suffices to consider the case of two rotor configurations
$r$ and $r'$ that differ only at a single vertex $x$, at which
the rotor is incremented once in $r'$ compared with $r$.
Suppose that $r$ is recurrent, and start a rotor walk at $x$.
At the first step, the rotor configuration becomes $r'$, and
the particle moves to a neighbour $y$, say.  Hence the rotor
walk started at $y$ is recurrent for $r'$, i.e.\ $r'$ is
recurrent.
\end{proof}

\section{Recurrent configurations}

Theorems \ref{T:zd}--\ref{T:any} are consequences of the more
general result below. For a set of vertices $S$, let $\partial
S$ denote its outer vertex-boundary:
\[
\partial S := \{x\in S^C: x \text{ has a neighbour in }S\}.
\]
We say that $S$ has \df{reflecting boundary} (for a given rotor mechanism and
initial rotor configuration) if for every vertex $y$ in $\partial S$, the
rotor at $y$ will send the particle to each of $y$'s neighbours in $S$ before
sending it to any other neighbour of $y$.

\begin{prop}\label{reflect}
If for some rotor configuration on a graph $G$, every finite
set of vertices is a subset of some finite set with reflecting
boundary, then the rotor walk starting from any vertex is
recurrent.
\end{prop}

To prove the above result, it will sometimes be convenient to
identify $G$ with the directed graph in which each undirected
edge is replaced with two directed edges, one in each
direction. We say that the rotor walk traverses the directed
edge $(x,y)$ when the particle makes a step from vertex $x$ to
vertex $y$.

\begin{lemma}\label{twice}
  If a rotor walk started at $x$ traverses some directed edge twice, then
  the first directed edge to be traversed twice is from $x$. In particular
  such a walk has returned to $x$.
\end{lemma}

\begin{proof}
  Let $(y,z)$ be the first directed edge to be traversed twice. At that
  time $y$ has sent the particle to all other neighbours exactly once, thus
  $y$ has emitted the particle $d_y+1$ times. Since no other edge has been
  traversed twice, $y$ has received the particle at most $d_y$ times. Thus
  $y=x$.
\end{proof}

\begin{lemma}\label{return}
  If $S$ is a set with reflecting boundary, then the rotor walk started at a
  vertex $x\in S$ will return to $x$ before leaving $S\cup \partial S$.
\end{lemma}

\begin{proof}
  We prove the stronger statement that the walk will return to $x$ before
  making any step from $\partial S$ to $S^C$. Consider the first such
  step, from some $y\in \partial S$ to some $z\not\in S$. At this time,
by the
  definition of reflecting boundary, $y$ has previously sent the particle
  to each of its neighbours in $S$. However, by our assumption it has
  received the particle only from its neighbours in $S$, therefore it must
  have received it twice from some such neighbour. Now apply
  Lemma~\ref{twice}.
\end{proof}

\begin{proof}[Proof of Proposition~\ref{reflect}]
  Suppose the particle started at $x$, and is currently at $x$, and let $A$
  be the (finite) set of vertices that have been visited. Then $A\subseteq
  S$ for some $S$ which had reflecting boundary in the initial rotor
  configuration. Since $\partial S$ has not yet been visited, $S$ still has
  reflecting boundary. By Lemma~\ref{return}, the walk will return again to
  $x$ (before leaving $S\cup \partial S$). Iterating this shows that the
  walk is recurrent.
\end{proof}

\begin{proof}[Proof of Theorem~\ref{T:zd}]
  By Proposition \ref{reflect}, it suffices to choose an initial rotor
  configuration so that every cube of the form $[-n,n]^d$
  ($n=0,1,2,\dots$) has reflecting boundary. Since each vertex $z$ of
  $\partial ([-n,n]^d)$ has only one neighbour $y$ in $[-n,n]^d$, this is
  achieved by setting the rotor at $z$ so that it will next point to $y$.
  Since the boundaries of different cubes are disjoint, this can be done
  for all $z$ and $n$.
\end{proof}

\begin{proof}[Proof of Theorem~\ref{T:any}]
  Let $S_0$ be any finite non-empty set of vertices, and construct
  $S_1,S_2,\dots$ inductively by $S_{i+1}:=S_i\cup \partial S_i$. By
  Proposition \ref{reflect}, it suffices to choose the rotor mechanism and
  configuration so that each $S_i$ has reflecting boundary. This is clearly
  possible, since the sets $\partial S_0,\partial S_1,\dots$ are disjoint.
\end{proof}

To prove Theorem \ref{T:planar}, we need a lemma about planar
graphs. Given a planar embedding of a planar graph, the edges
incident to a given vertex fall in some cyclic order around it.
We say that a set of vertices $S$ is \df{pincerless} if for
every $x\in\partial S$, either all neighbours of $x$ lie in
$S$, or the incident edges joining $x$ to $S$ lie in one
contiguous interval in the cyclic order around $x$.

\begin{lemma}\label{smooth}
Let $G$ be an infinite, connected, planar, simple graph with
all degrees finite, and with a locally finite planar embedding.
Any finite set of vertices is a subset of some finite
pincerless set.
\end{lemma}

\begin{proof}[Proof of Lemma~\ref{smooth}]
Let $A$ be a finite set; we will show that it is a subset of
some finite pincerless $B$.  By enlarging $A$ if necessary, we
may assume that $A$ is connected (i.e.\ it induces a connected
subgraph of $G$).

Consider any $x\in \partial A$.  Each neighbour of $x$ lies
either in $A$, or in a finite or an infinite component of
$(A\cup\{x\})^C$.  We call the three types $A$-neighbours,
$F$-neighbours, and $I$-neighbours, respectively.  Note that
there is at least one $A$-neighbour.  We claim that it is
impossible for two $A$-neighbours $a,a'$ and two $I$-neighbours
$i,i'$ to alternate in the cyclic order of neighbours of $x$
(i.e.\ to occur in the order $aia'i'$ when we delete all other
neighbours from the order).  This follows from planarity and
local finiteness, because $a$ and $a'$ are connected by a path
in $A( \not\ni x)$, while $i$ and $i'$ are each connected to
infinity off $A\cup\{x\}$.  Therefore, there exists some
interval in the cyclic order of neighbours of $x$ that contains
all the $A$-neighbours and no $I$-neighbours.   Let $J_x$ be
the unique minimal such interval, if there is at least one
$I$-neighbour, and otherwise the set of all $x$'s neighbours.
Define $D_x$ to be the union of all the finite components of
$(A\cup\{x\})^C$ corresponding to $F$-neighbours in $J_x$.
Define
$$B:= A\cup \bigcup_{x\in\partial A} D_x.$$

The set $B$ is clearly finite, and we must check that it is pincerless.
First, we claim that if a vertex $z$ is adjacent to some vertex in $D_y$ for
some $y\in\partial A$ with $y\neq z$, then $z\in B$. This follows because
either $z\in A\subseteq B$, or $z\in (A\cup\{y\})^C$, in which case $z$ lies
in the same component of the latter set as does its neighbour in $D_y$, so
since $D_y$ is a union of such components, $z\in D_y\subseteq B$.

Now consider any $x\in\partial B$.  We must have $x\in\partial
A$, otherwise the above claim would imply $x\in B$.  Now by the
definition of $D_x$, all of $x$'s neighbours in the interval
$J_x$ lie in $B$.  Therefore to check the pincerless condition
at $x$ it suffices to show that no neighbour of $x$ not in
$J_x$ lies in $B\setminus A$.  By the definition of $B$, such a
neighbour would lie in $D_y$ for some $y\neq x$, and hence by
the claim again we would have $x\in B$, a contradiction.
\end{proof}

\begin{proof}[Proof of Theorem \ref{T:planar}]
  Let $S_0$ be any finite non-empty set of vertices, and define
  $S_1,S_2,\ldots$ inductively by taking $S_{i+1}$ to be a finite pincerless
  set containing $S_i\cup\partial S_i$, by Lemma \ref{smooth}. For $i\geq
  1$, since $S_i$ is pincerless, and the rotors rotate clockwise or
  anticlockwise, the rotors in $\partial S_i$ can be given initial
  directions so that $S_i$ has reflecting boundary. Since the $\partial
  S_i$ are disjoint, this can be done for all $i\geq1$ simultaneously.
\end{proof}

\section{Exact number of visits}
\label{S:exact}

We now turn to Theorem~\ref{T:zd-exact}, which is a consequence
of the more general result below.  Given a rotor configuration
and some set of vertices $U\subseteq V(G)$, we associate a
directed graph with vertex set $V$, and a directed edge from
$x$ to $y$ whenever $x\in U$ and the rotor at $x$ points to
$y$.

\begin{prop}\label{exact}
  Let $S_0,S_1,\ldots$ be finite sets such that $S_0=\{a\}$, and $S_i\cup
  \partial S_i \subseteq S_{i+1}$ for all $i$. Suppose that in the initial
  rotor configuration, for each $i\geq 0$:
  \begin{ilist}
  \item $S_i$ has reflecting boundary, and
  \item the rotors at vertices of $S_{i+1} \setminus (S_i\cup\partial S_i)$
    form a directed forest pointing towards $S_i\cup \partial S_i$.
  \end{ilist}
Consider the rotor walk started at $a$, and let $b$ be the next
vertex it visits. Then for every $k\geq 1$, in the time
interval from the $k$th to the $(k+1)$st traversal from $a$ to
$b$ (inclusive and exclusive respectively), the rotor walk
traverses each edge incident to $S_k$ exactly once in each
direction, and no other edges.
\end{prop}

\begin{proof}
  By a \df{unicycle rooted at $x$} we mean a directed graph comprising a
  single oriented cycle passing through vertex $x$, together with a
  collection of directed trees rooted on the cycle and pointing towards it.
  We will use the following fact about the rotor walk on a {\em finite}
  undirected graph (or indeed an Eulerian directed graph); see e.g.\
  \cite[Lemma~4.9]{hlmppw}. Starting from any rotor configuration that
  forms a unicycle rooted at the current particle location, if $2m$ rotor
  steps are performed, where $m$ is the number of edges of the graph, then
  each edge is traversed exactly once in each direction, and
  the rotors and particle finish in their initial positions.

The rotor walk is recurrent by Proposition \ref{reflect}. Let
$t_k$ be the time just before the $(k+1)$st traversal from $a$
to $b$, so $t_0=0$, and at time $t_k$ the particle is at $a$.
We will prove by induction that for all $k\geq 1$: at time
$t_k$, no vertex outside $S_k$ has been visited, and the rotors
in $S_k$ form a unicycle rooted at $a$. In the process of
proving this we will establish the claim of the proposition.

First consider the time period from $t_0$ to $t_1$. Since $S_0=\{a\}$ has
reflecting boundary, the particle simply traverses each edge incident to $a$
in each direction, so the claim of the proposition holds for $k=0$. At time
$t_1$, the rotors of $\partial S_0$ all point towards $a$, so the rotors of
$S_0\cup \partial S_0$ form a unicycle rooted at $a$. Since the vertices of
$S_1\setminus (S_0\cup\partial S_0)$ have not been visited, condition (ii) of
the proposition thus implies that the rotors of $S_1$ form a unicycle rooted
at $a$, establishing the inductive hypothesis for $k=1$.

Now suppose the inductive hypothesis holds for some $k\geq 1$.
Consider the finite subgraph $F_k$ of $G$ comprising all edges
incident to $S_k$, with vertex set $S_k\cup\partial S_k$.
Consider a rotor configuration on $F_k$ given as follows.  Let
the vertices in $S_k$ inherit their positions from $G$ at time
$t_k$. Recalling that $S_k$ has reflecting boundary in $G$, fix
the rotor at each $y\in\partial S_k$ so that it will next send
the particle to each of $y$'s neighbours in $S_k$, in the same
order as in $G$; thus, the rotor should initially point to the
last of these neighbours. Since the rotors of $S_k$ form a
unicycle rooted at $a$, the same applies to $F_k$. Therefore,
running the rotor walk on $F_k$, started at $a$, until just
before the second traversal from $a$ to $b$ results in each
edge of $F_k$ being traversed once in each direction, and the
same final rotor configuration. Since the rotors in $\partial
S_k$ have only sent the particle towards $S_k$, the behaviour
of the walk on $G$ over this time interval is identical.  Hence
the claim of the proposition holds for $k$. At time $t_{k+1}$,
the rotors in $S_k\cup\partial S_k$ form a unicycle rooted at
$a$, and condition (ii) again implies that the same holds when
all rotors in $S_{k+1}$ are included, establishing the
inductive hypothesis for $k+1$.
\end{proof}

\begin{proof}[Proof of Theorem \ref{T:zd-exact}]\sloppypar
  We apply Proposition~\ref{exact}, with $S_i=[-i,i]^d$. The boundary
  $\partial S_i$ consists precisely of those vertices $x$ of
  $S_{i+1}\setminus S_i$ that have $\alpha(x)=1$ (i.e.\ those on
  $(d-1)$-dimensional faces of $S_{i+1}\setminus S_i$). We set the rotors at these
  vertices to point {\em next} towards $S_i$, and those of
  $S_{i+1}\setminus(S_i\cup\partial S_i)$ to form a forest pointing towards
  $S_i\cup\partial S_i$, as required. By Proposition~\ref{exact}, between
  times $t_k$ and $t_{k+1}$, each vertex of $S_k$ is visited $2d$ times,
  and each vertex of $\partial S_k$ is visited once. Summing over $k$ gives
  the claimed expression.
\end{proof}

\section*{Open Questions}

\begin{ilist}
\item Do there exist a graph and rotor mechanism for which
    every initial rotor configuration is transient?
\item While Theorem~\ref{T:zd-exact} provides a detailed
    description of the behaviour of certain recurrent rotor
    configurations, some simple transient examples remain mysterious.
    On $\Z^2$, let each rotor initially point East and rotate anticlockwise.  Start
    a rotor walk at $0$, and restart it at $0$ after each escape
    to infinity.  What is the asymptotic growth rate of the
    number of escapes to infinity prior to the $n$th visit to
    $0$, as $n\to\infty$?  (By a result of Schramm,
    \cite[Theorem~10]{h-propp}, it is $o(n)$).
\end{ilist}

\bibliographystyle{habbrv}
\bibliography{bib}

\vspace{2mm} \noindent
{\sc Omer Angel:} {\tt angel at math.ubc.ca} \\
Department of Mathematics, University of British Columbia, Vancouver BC V6T
1Z2, Canada

\vspace{2mm} \noindent
{\sc Alexander E. Holroyd:} {\tt holroyd at microsoft.com} \\
Microsoft Research, 1 Microsoft Way, Redmond WA 98052, USA

\end{document}